\documentclass{amsart}
\usepackage{amsmath,amssymb,latexsym, amsfonts, amscd, amsthm, mathrsfs}

\usepackage{xcolor}
\usepackage{hyperref}

\hypersetup{colorlinks=false,linkbordercolor=red,linkcolor=green,pdfborderstyle={/S/U/W 1}}

\usepackage{bbm}
\usepackage[cmtip,all]{xy}
\usepackage{enumerate}
\usepackage{mathdots}

\numberwithin{equation}{section}

\newtheorem{thm}{Theorem}[section]
\newtheorem{pro}[thm]{Proposition}
\newtheorem{lm}[thm]{Lemma}
\newtheorem{cor}[thm]{Corollary}

\theoremstyle{definition}

\theoremstyle{remark}
\newtheorem{rem}[thm]{Remark}

\numberwithin{equation}{section}

\DeclareMathOperator*{\End}{End}
\DeclareMathOperator*{\el}{ell}
\DeclareMathOperator*{\Irr}{Irr}
\DeclareMathOperator*{\disc}{disc}
\DeclareMathOperator*{\temp}{temp}
\DeclareMathOperator*{\Gal}{Gal}
\DeclareMathOperator*{\Res}{Res}

\DeclareMathOperator*{\ad}{ad}
\DeclareMathOperator*{\diag}{diag}

\DeclareMathOperator*{\der}{der}
\DeclareMathOperator*{\Hom}{Hom}

\DeclareMathOperator*{\Sp}{Sp}

\DeclareMathOperator*{\SU}{SU}

\DeclareMathOperator*{\U}{U}
\DeclareMathOperator*{\SO}{SO}

\DeclareMathOperator*{\SL}{SL}
\DeclareMathOperator*{\GL}{GL}

\DeclareMathOperator*{\PGL}{PGL}

\DeclareMathOperator*{\Mat}{Mat}

\newcommand{\s}{\simeq}

\newcommand{\tsigma}{\Sigma}
\newcommand{\ttau}{\Upsilon}

\newcommand{\tphi}{\varphi}

\newcommand{\CC}{\mathbb{C}}
\newcommand{\NN}{\mathbb{N}}

\newcommand{\QQ}{\mathbb{Q}}
\newcommand{\ZZ}{\mathbb{Z}}

\newcommand{\ii}{\mathbf{\textit{i}}}

\def\bA{\bold A}
\def\bG{\bold G}
\def\bP{\bold P}
\def\bM{\bold M}
\def\bN{\bold N}

\def\sC{\mathcal C}

\newcommand{\tG}{\mathcal{G}}
\newcommand{\tbG}{\boldsymbol{\mathcal{G}}}
\newcommand{\tP}{\mathcal{P}}

\newcommand{\tM}{\mathcal{M}}
\newcommand{\tbM}{\boldsymbol{\mathcal{M}}}

\title[Behavior of $R$-groups for $p$-adic inner forms of $SU_n$]{Behavior of $R$-groups for $p$-adic inner forms of quasi-split special unitary groups} 

\author[Kwangho Choiy and David Goldberg]{Kwangho Choiy and David Goldberg} 
\thanks{}

\email{kchoiy@gmail.com}

\email{david77@gmail.com}


\begin{document}
\maketitle

\begin{center}
{\em For Freydoon, with deep gratitude for all he has done for us, and for the field. }
\end{center}

\begin{abstract}

We study $R$-groups for $p$-adic inner forms of quasi-split special unitary groups. We prove Arthur's conjecture, the isomorphism between the Knapp-Stein $R$-group and the Langlands-Arthur $R$-group, for quasi-split special unitary groups and their inner forms. Furthermore, we investigate the invariance of the Knapp-Stein $R$-group within $L$-packets and between inner forms. This work is applied to transferring known results in the second-named author's earlier work for quasi-split special unitary groups to their non quasi-split inner forms.
\end{abstract}
\section{Introduction} \label{intro}
In the representation theory of $p$-adic groups, in particular, in the framework of the local Langlands correspondence for a connected reductive algebraic group $\bG$ over a $p$-adic field $F,$ it is of great importance to study the reducibility of parabolically induced representations.
This study yields information on constructing tempered $L$-packets of $\bG$ from discrete $L$-packets of its $F$-Levi subgroup.
The determination of reducibility has been developing over decades via several approaches, for example, by means of harmonic analysis from investigation of poles or zeros of residues of intertwining operators, Plancherel measures, and local $L$-functions, \cite{sh90, sh95, sil79}. 
The method we address here is in terms of the Knapp-Stein $R$-group, which provides a combinatorial description of the tempered dual of $\bG(F)$ as well as its elliptic tempered spectrum.
Further, as conjectured by Arthur, the isomorphism of the Knapp-Stein and Langlands-Arthur $R$-groups, via the endoscopic $R$-group, plays a significant role in the comparison of trace formulas and the endoscopic classification of automorphic representations \cite{art12, kmsw14, mok13}.

While there has been a great deal of progress on the theory of Knapp-Stein $R$-groups for $F$-quasi-split groups $\bG,$ little is known for non $F$-quasi-split groups $\bG'.$ 
In \cite{chgo12}, we investigated the behavior of $R$-groups between $F$-inner forms of $\SL_n,$ and determined the Knapp-Stein $R$-groups for the non quasi-split inner form from those of the split $\SL_n.$
We also proved the Knapp-Stein $R$-group for $\bG'(F)$ embeds as a subgroup of the $R$-group for $\SL_n,$ and we characterized the quotient.
Another approach to this case was carried out in \cite{chaoli} and an example was discovered for which the Knapp-Stein $R$-group for $\bG'(F)$ is strictly smaller than that of $\SL_n.$
In \cite{cgclassical}, we further showed the invariance of $R$-groups between $F$-inner forms of quasi-split classical groups $\SO_{2n+1},$ $\Sp_{2n},$ $\SO_{2n},$ or $\SO_{2n}^*,$ and transferred all known, relevant facts developed by the second-named author in the quasi-split classical groups  to their non quasi-split inner forms.  
Non quasi-split inner forms of $\Sp_{4n}$ and $\SO_{4n}$ are also treated in \cite{han04}. 

To study the Knapp-Stein $R$-group for a non quasi-split group $\bG'$ in general, as one may notice from our previous works \cite{chgo12, cgclassical}, it is natural that we investigate its behavior between $\bG$ and $\bG'$ and transfer the developed theories regarding the Knapp-Stein $R$-group from $\bG$ to $\bG'.$ 
This is due to the notion of Langlands functoriality.
Since $\bG'$ is induced from a new Galois action twisted by a Galois $1$-cocycle in the inner automorphisms of $\bG,$ their $L$-groups are isomorphic over the Galois group and the set of $L$-parameters for $\bG'$ is contained in that for $\bG.$
In this paper, combining this strategy and the restriction method, we study the Knapp-Stein $R$-group for non quasi-split $F$-inner forms of quasi-split special unitary groups.

More precisely, 
we fix a quadratic extension $E$ of a $p$-adic field $F$ of characteristic zero. 
Let $\bG_n = \SU_n$ be a quasi-split special unitary group over $F$ with respect to $E/F$ and let $\bG'_n$ be its non quasi-split inner form over $F$. 
A simple consequence from the Satake classification or a computation of Galois cohomology reduces our study to the case when $n$ is even. In fact, there is a unique non quasi-split $F$-inner form $\bG'_n,$ up to $F$-isomorphism (see Section \ref{structure}). 
For the rest of the introduction, we assume that $n$ is even, unless otherwise stated. Let $\bM'$ be an $F$-Levi subgroup of $\bG'_n,$ which is an $F$-inner form of an $F$-Levi subgroup $\bM$ of $\bG_n.$ Then, $\bM = \tbM \cap \bG_n$ and $\bM' = \tbM' \cap \bG'_n,$ where $\tbM$ is an $F$-Levi subgroup of a quasi-split unitary group $\tbG_n = \U_n$ over $F$ with respect to $E/F$ and $\tbM'$ is an $F$-Levi subgroup of a non quasi-split $F$-inner form $\tbG'_n$ of $\tbG_n.$ 
We shall use $G$ for the group  $\bG(F)$ of $F$-points of any connected reductive algebraic group $\bG$ over $F.$

Given an elliptic tempered $L$-parameter $\phi \in \Phi_{\disc}(M),$ 
by \cite[Th\'{e}or\`{e}m 8.1]{la85}, we have a lifting $\tphi \in \Phi_{\disc}(\tM)$ commuting with the natural projection $\widehat{\tM} \twoheadrightarrow \widehat{M},$ 
where $\widehat{\tM}$ and $\widehat{M}$ respectively denote the connected components of the $L$-groups of $\tbM$ and $\bM$ (see Section \ref{pre} for the details).
Restricting the $L$-packet $\Pi_{\tphi}(\tM)$ constructed by Rogawski \cite{rog90} and Mok \cite{mok13} 
(see Section \ref{LLC for unitary}),
we construct an $L$-packet  $\Pi_{\phi}(M)$ as the set of isomorphism classes of irreducible constituents in the restriction from $\tM$ to $M.$
All the arguments used for $\bM$ can be applied to  $\bM',$ as in Kaletha-Minguez-Shin-White \cite{kmsw14}
and the details are described in Section \ref{L-packets}.
For any $\sigma \in \Pi_{\phi}(M),$ and $\sigma' \in \Pi_{\phi}(M'),$ we prove 
\[
R_{\sigma} \s R_{\phi, \sigma} ~~ \text{and} ~~R_{\sigma'} \s R_{\phi, \sigma'}.
\]
In each of the above isomorphisms, the left side is the Knapp-Stein $R$-group, and and the right side is the Langlands-Arthur $R$-group, defined in Section \ref{section for def of R}. 
This is known as Arthur's conjecture, predicted in \cite{art89ast}, for $\bG$ and $\bG'$ (See Theorem \ref{arthur conj}).
In the course of the proofs, we apply some known results about $R$-groups for
$\tbG_n$ and $\tbG'_n$ 
in \cite{go95, kmsw14, mok13}, which are recalled in Section \ref{R-groups for U(n) and its inner forms}.
We also investigate and utilize some relationships between identity components of the centralizers of the images of $\tphi$ and $\phi$ in $\widehat G$ and $\widehat \tG$ (see Section \ref{L-groups} and Lemma \ref{lm for identity comp}).

We further study the invariance of the Knapp-Stein $R$-groups for $\bG_n$ and $\bG'_n.$ 
Namely, given $\sigma_1, ~ \sigma_2 \in \Pi_{\phi}(M),$ and $\sigma'_1, ~ \sigma'_2 \in \Pi_{\phi}(M'),$ we prove
$
R_{\sigma_1} \s R_{\sigma_2}, ~~ \text{and}~~
R_{\sigma'_1} \s R_{\sigma'_2}
$
(Theorem \ref{invariance within L-packet}).
Moreover, given $\sigma \in \Pi_{\phi}(M)$ and $\sigma' \in \Pi_{\phi}(M),$ we prove
$
R_{\sigma} \s R_{\sigma'}
$
(Theorem \ref{invariance 2 within L-packet}).
The crucial idea is to study the stabilizers $W(\sigma)$ and $W(\sigma').$ 
Theorem \ref{equality on W} shows
\begin{equation} \label{iso intro}
W(\sigma) \s 
\{
w \in W_M : {^w}\tsigma \s \tsigma \lambda ~~ \text{ for some}~ \lambda \in (\tM/M)^D
\},
\end{equation}
where $W_M$ is the Weyl group of $\bM$ in $\bG,$ 
$\tsigma \in \Pi_{\tphi}(\tM)$ is a lifting of $\sigma,$ and $(\tM/M)^D$ is the group of continuous characters on $\tM$ which are trivial on $M.$
The same is true for $W(\sigma').$
The isomorphism \eqref{iso intro} stems from the group structure of $M$ and $M'$ and the description of irreducible representations of $M$ and $M',$ which are discovered in \cite[Section 2]{go95} and explained in Section \ref{behavior of Rgps}.

As an application, a combination of Theorem \ref{invariance 2 within L-packet} and the second-named author's earlier result \cite[Theorem 3.7]{go95} shows that the Knapp-Stein $R$-group, $R_{\sigma'},$ for non quasi-split inner forms $\bG'$ can be expressed in terms of a subgroup in $R_{\sigma'}$ and $\ZZ^d$ for some integer $d.$
It follows that $R_{\sigma'}$ is in general non-abelian (see Remark \ref{final remark}).

In Section \ref{pre}, we recall basic notation and background, provide the detailed group structure of $\tbG_n,$ $\bG_n,$ $\tbG'_n,$ $\bG'_n,$ and their $F$-Levi subgroups, and study relations in their $L$-groups.
In Section \ref{weyl group actions}, we describe Weyl group actions on Levi subgroups and their representations. In Section  \ref{R-groups for U(n) and its inner forms}, we revisit the theory of $R$--groups for $\tbG_n$ and $\tbG_n'$ based on \cite{go95, kmsw14, mok13}.
In Section \ref{R-groups for SU(n) and its inner forms}, we prove the Arthur's conjecture for $\bG_n$ and $\bG_n'$ and the invariance of their $R$-groups.

\subsection*{Acknowledgements}  
K. Choiy was partially supported by an AMS-Simons Travel Grant.
D. Goldberg was partially supported by  Simons Foundation Collaboration Grant 279153.

\section{Preliminaries} \label{pre}
\subsection{Basic notation and background} \label{basic notation}
Let $p$ be a prime.  We let $F$ be a $p$-adic field of characteristic $0,$ i.e., a finite extension of $\QQ_{p}.$  
Fix an algebraic closure $\bar{F}$ of $F.$ 
Given a connected reductive algebraic group $\bG$ defined over $F,$  we write $\bG(F)$ for the group of $F$-points. Unless otherwise stated, we shall use $G$ for $\bG(F)$ and likewise for other algebraic groups defined over $F.$

Fix a minimal $F$-parabolic subgroup $\bP_0$ of $\bG$ with Levi decomposition $\bP_0=\bM_0 \bN_0,$ where $\bM_0$ denotes a Levi factor and $\bN_0$ denotes the unipotent radical. 
We denote by $\bA_0$ the split component of $\bM_0,$ that is, the maximal $F$-split torus in the center of $\bM_0,$ and by $\Delta$ the set of simple roots of $\bA_0$ in $\bN_0.$ 
We say an $F$-parabolic subgroup $\bP$ of $\bG$ to be standard if it contains $\bP_0.$

Given an $F$-parabolic subgroup $\bP$ with Levi decomposition $\bP=\bM\bN,$ there exist a subset $\Theta \subseteq \Delta$ such that $\bM$ equals $\bM_{\Theta},$ the Levi subgroup generated by $\Theta.$
Note that $\bM \supseteq \bM_0$ and $\bN \subseteq \bN_0.$
We write  $\bA_{\bM_{\Theta}}$ for the split component $\bA_{\bM}$ of $\bM=\bM_{\Theta}.$ 
It follows that $\bA_{\bM}$ equals the identity component
$(\cap_{\alpha \in \Theta} \ker \alpha
)^{\circ}$ in $\bA_0,$
so that $\bM = Z_\bG(\bA_\bM),$ 
where $Z_\bG(\bA_\bM)$ is the centralizer of $\bA_\bM$ in $\bG.$ 
We refer the reader to \cite[Proposition 20.4]{borel91} and \cite[Section 15.1]{springer98} for the details. 

We let $\Phi(P, A_M)$ denote the set of reduced roots of $\bP$ with respect to $\bA_\bM.$ 
Denote by $W_M = W(\bG, \bA_\bM) := N_\bG(\bA_\bM) / Z_\bG(\bA_\bM)$ the Weyl group of $\bA_\bM$ in $\bG,$ 
where $N_\bG(\bA_\bM)$ is the normalizer of $\bA_\bM$ in $\bG.$ 
For simplicity, we write $\bA_0 = \bA_{\bM_0}.$ 

For any topological group $H,$ we write $Z(H)$ for the center of $H.$ Denote by $H^\circ$ the identity component of $H$ and by $\pi_0(H)$ the group $H/H^\circ$ of connected components of $H.$ 
We write $H^D$ for the group, $\Hom(H , \CC^{\times}),$ of all continuous characters and by $\mathbbm{1}$ the trivial character.
We say a character is unitary if its image is in the unit circle $S^1 \subset \CC^{\times}.$ 
For any Galois module $J,$ we denote by $H^i(F, J) := H^i(\Gal (\bar{F} / F), J(\bar{F}))$ 
the Galois cohomology of $J$ for $i \in \NN.$

Let $\bG$ be a connected reductive algebraic group  over $F.$ 
We denote by $\Irr(G)$ the set of isomorphism classes of irreducible admissible complex representations of $G.$ 
If there is no confusion, we do not make a distinction between each isomorphism class and its representative. For any $\sigma \in \Irr(M),$ we write $\ii_{G,M} (\sigma)$ for the normalized (twisted by $\delta_{P}^{1/2}$) induced representation, where $\delta_P$ denotes the modulus character of $P.$ 
We denote by $\sigma^{\vee}$ the contragredient of $\sigma.$

Denote by $\Pi_{\disc}(G)$ and $\Pi_{\temp}(G)$  the subsets of $\Irr(G)$ which respectively consist of discrete series and tempered representations,
where, a discrete series representation is an irreducible, admissible, unitary representation whose matrix coefficients are square-integrable modulo the center of $G,$ that is, in $L^2(G/Z(G)),$
and a tempered representation is an irreducible, admissible, unitary representation whose matrix coefficients are in $L^{2+\epsilon}(G/Z(G))$ for all $\epsilon > 0.$
It is clear that $\Pi_{\disc}(G) \subset \Pi_{\temp}(G).$

We denote by $W_F$ the Weil group of $F$ and by $\Gamma$ the absolute Galois group $\Gal(\bar{F} / F).$ 
By fixing $\Gamma$-invariant splitting data, 
the $L$-group of $G$ is defined as a semi-direct product $^{L}G := \widehat{G} \rtimes W_F$ (see \cite[Section 2]{bo79}).
Following \cite[Section 8.2]{bo79}, an $L$-parameter for $G$ is an admissible homomorphism 
\[
\phi: W_F \times SL_2(\CC) \rightarrow {^L}G.
\]
Two $L$-parameters are said to be equivalent if they are conjugate by $\widehat{G}.$
We denote by $\Phi(G)$ the set of equivalence
classes of $L$-parameters for $G.$
 
We denote by $C_{\phi}(\widehat{G})$ the centralizer of the image of $\phi$ in $\widehat{G}.$
The center of $^{L}G$ is the $\Gamma$-invariant group $Z(\widehat{G})^{\Gamma}.$ 
Note that $C_{\phi} \supset Z(\widehat{G})^{\Gamma}.$ 
We say that $\phi$ is elliptic if the quotient group 
$C_{\phi}(\widehat G) / Z(\widehat{G})^{\Gamma}
$
is finite, and $\phi$ is tempered if $\phi(W_F)$ is bounded. 
We denote by $\Phi_{\el}(G)$ and $\Phi_{\temp}(G)$ the subset of $\Phi(G)$ which respectively consist of elliptic and tempered $L$-parameters of $G.$ 
We set $\Phi_{\disc}(G) = \Phi_{\el}(G) \cap \Phi_{\temp}(G).$

The local Langlands conjecture for $G$ predicts that there is a surjective finite-to-one map from $\Irr(G)$ to $\Phi(G).$
Given $\phi \in \Phi(G),$ 
we write $\Pi_{\phi}(G)$ for the $L$-packet attached to $\phi,$ and then the local Langlands conjecture implies that
\[
\Irr(G) = \bigsqcup_{\phi \in \Phi(G)}\Pi_{\phi}(G).
\] 
It is expected that $\Phi_{\disc}(G)$ and $\Phi_{\temp}(G)$ respectively parameterize $\Pi_{\disc}(G)$ and $\Pi_{\temp}(G).$

Given two connected reductive algebraic groups $\bG$ and $\bG'$ over $F,$ 
 $\bG'$ is said to be an \textit{$F$-inner form} of $\bG$ with respect to an $\bar{F}$-isomorphism $\varphi: \bG' \overset{\sim}{\rightarrow} \bG$ if $\varphi \circ \tau(\varphi)^{-1}$ is an inner automorphism ($g \mapsto xgx^{-1}$) defined over $\bar{F}$ for all $\tau \in \Gal (\bar{F} / F)$ (see \cite[2.4(3)]{bo79} or \cite[p.280]{kot97}).
If there is no confusion, we often omit the references to $F$ and $\varphi.$ 
It is well known from \cite[p.280]{kot97} that there is a bijection between $H^1(F, \bG_{\ad})$ and the set of isomorphism classes of $F$-inner forms of $\bG,$ where $\bG_{\ad}:=\bG / Z(\bG).$
We note that, when $\bG$ and $\bG'$ are inner forms of each other, we have $^L{G} \s {^LG'}$ \cite[Section 2.4(3)]{bo79}. 
\subsection{$R$-groups} \label{section for def of R}
We review the definitions of Knapp-Stein, Langlands-Arthur, and endoscopic $R$--groups.
Let $\bG$ be a connected reductive algebraic group over $F,$ and let $\bM$ be an $F$-Levi subgroup of $\bG.$
Given $\sigma \in \Irr(M)$ and $w \in W_M,$ we write ${^w}\sigma$ for the representation given by ${^w}\sigma(x)=\sigma(w^{-1}xw).$ 
Note that, for the purpose of studying $R$--groups, we do not distinguish the representative of $w,$ since the isomorphism class of ${^w}\sigma$ is independent of the choices of representatives in $G$ of $w \in W_M.$
Assume that $\sigma$ lies in $\Pi_{\disc}(M),$ we define the stabilizer of $\sigma$ in $W_M$
\[ W(\sigma) := \{ w \in W_M : {^w}\sigma \s \sigma \}.
\]
Write $\Delta'_\sigma$ for $\{ \alpha \in \Phi(P, A_M) : \mu_{\alpha} (\sigma) = 0 \},$ where $\mu_{\alpha} (\sigma)$ is the rank one Plancherel measure for $\sigma$ attached to $\alpha$ \cite[p.1108]{goldberg-class}. 
The \textit{Knapp-Stein $R$-group} is defined by
\[
R_{\sigma} := \{ w \in W(\sigma) : w \alpha > 0, \; \forall \alpha \in \Delta'_\sigma \}.
\]
We denote by $W^{\circ}_{\sigma}$ the subgroup of $W(\sigma),$ generated by the reflections in the roots of $\Delta'_\sigma.$
Then, for any $\sigma \in \Pi_{\disc}(M),$ we have 
\[
W(\sigma) = R(\sigma) \ltimes W^{\circ}_{\sigma},
\]
which yields another description of the Knapp-Stein $R$-group
\[
R(\sigma) \s  W(\sigma)/W^{\circ}_{\sigma}.
\]
We refer to \cite{ks72, sil78, sil78cor} for details.

Given an $L$-parameter $\phi \in \Phi(M),$
we also consider $\phi$ as an $L$-parameter for $G$ via the inclusion $\widehat M \hookrightarrow \widehat G.$
Fix a maximal torus $T_{\phi}$ in $C_{\phi}(\widehat{G})^{\circ}.$ 
We set
\[
W_{\phi}^{\circ} := N_{C_{\phi}(\widehat{G})^{\circ}} (T_{\phi}) /  Z_{C_{\phi}(\widehat{G})^{\circ}} (T_{\phi}), \quad W_{\phi} := N_{C_{\phi}(\widehat{G})} (T_{\phi}) /  Z_{C_{\phi}(\widehat{G})} (T_{\phi}).
\]
The \textit{endoscopic $R$-group} $R_{\phi}$ is defined as follows
\[
R_{\phi}:=W_{\phi}/W_{\phi}^{\circ}.
\]
We identify $W_{\phi}$ with a subgroup of $W_M$ (see \cite[p.45]{art89ast}). 
For $\sigma \in \Pi_{\phi}(M),$ we set
\begin{equation*} \label{def of W_phi, sigma}
W_{\phi, \sigma}^{\circ} :=  W_{\phi}^{\circ} \cap W(\sigma), \quad  W_{\phi, \sigma} :=  W_{\phi} \cap W(\sigma).
\end{equation*}
The \textit{Langlands-Arthur $R$-group} $R_{\phi, \sigma}$ is defined as follows
\[
R_{\phi, \sigma}:=W_{\phi, \sigma}/W_{\phi, \sigma}^{\circ}.
\]

\section{Structure theory of Levi subgroups} \label{structure}
We discuss the group structure of $\U_n,$ $\SU_n,$ their $F$-inner forms, and their $F$-Levi subgroups.
We mainly refer to \cite{go95, go06, kmsw14, mok13, sa71}.
Fix a quadratic extension $E/F$ with $\bar{~~}$ the non-trivial Galois element.
For any positive integer $n,$ we let
\[ 
J_n := 
\left( \begin{array}{ccccc}
 0& & & & 1\\
 & & & -1 &\\
 & & 1 & &\\ 
 & \iddots & & &\\ 
(-1)^n & & & &0 \\ 
 \end{array} \right) \in {\Mat}_{n \times n}(\ZZ).
\]   
We denote by $\Res_{E/F}$ the Weil restriction of scalars of $E/F$ (see \cite[Chapter 1]{weil82} and \cite[2.1.2]{pr94}). 
For $g=(g_{ij}) \in \Res_{E/F} \GL_n,$ we let $\bar{g}=(\bar g_{ij})$ and set $\varepsilon(g) = J_n {^t\bar{g}}^{-1} J_n^{-1},$ where $g \mapsto  {^t}{g}$ is the transpose. 
\subsection{$\U_n$ and its inner forms} \label{structure for Un}
Let $\tbG=\tbG_n$ denote the quasi-split unitary group $\U_{n}$ with respect to $E/F$ and $J_{n}.$ Thus, 
\[
\tbG = \{
g \in {\Res}_{E/F} {\GL}_n : gJ_n {^t\bar{g}} = J_n 
\}
\]
(see \cite[Section 1]{go06}; our $J_n$ is the inverse of $u_n$ therein).
We denote by $\tbM$ an $F$-Levi subgroup of $\tbG.$ Then, $\tbM$ is of the form
\begin{equation} \label{form of tbM}
{\Res}_{E/F} {\GL}_{n_1} \times \cdots \times {\Res}_{E/F} {\GL}_{n_k} \times {\tbG}_m,
\end{equation}
where $\sum_{i=1}^k 2n_i + m = n$ with $n_i \geq 0$ and $m \geq 0.$  
By convention, we note that, ${\tbG}_0 = 1$ for $n$ even, ${\tbG}_1 = \U_1$ for $n$ odd, and ${\GL}_0=1.$
The group of $F$-points $,A_{\tbM},$ of the split component $\bA_{\tbM}$ is of the form
\begin{equation} \label{type of AtM}
\{
{\diag} (x_1I_{n_1}, \cdots, x_k I_{n_k}, I_m, {x}_k^{-1}I_{n_k}, \cdots, {x}_1^{-1}I_{n_1}) : x_i \in F^{\times}
\}.
\end{equation} 

We let $\tbG'=\tbG'_n$ denote an $F$-inner form of $\tbG.$
By the Satake classification in \cite[Section 3.3]{sa71}, 
for $n$ \textbf{odd}, there is \textit{no non quasi-split} $F$-inner form of $\tbG_n.$
On the other hand, for $n$ \textbf{even},
there is a \textit{unique} non quasi-split $F$-inner form $\tbG'_n,$ up to $F$-isomorphism. 
The $\Gamma$-diagram of the connected, simply-connected, semi-simple type of such $\tbG'$ is 
\begin{equation} \label{diagram}
\xy
\POS (10,0) *\cir<2pt>{} ="a" ,
\POS (20,0) *\cir<2pt>{} ="b" ,
\POS (30,0) *\cir<0pt>{} ="e",
\POS (40,0) *\cir<0pt>{} ="f",
\POS (50,0) *\cir<2pt>{} ="h",
\POS (60,0) *\cir<2pt>{} ="i",
\POS (70,-5) *{\bullet} ="j",

\POS (10,-10) *\cir<2pt>{} ="a1" ,
\POS (20,-10) *\cir<2pt>{} ="b1" ,
\POS (30,-10) *\cir<0pt>{} ="e1",
\POS (40,-10) *\cir<0pt>{} ="f1",
\POS (50,-10) *\cir<2pt>{} ="h1",
\POS (60,-10) *\cir<2pt>{} ="i1",

\POS (10,-2) *\cir<0pt>{} ="a2" ,
\POS (20,-2) *\cir<0pt>{} ="b2" ,
\POS (50,-2) *\cir<0pt>{} ="h2",
\POS (60,-2) *\cir<0pt>{} ="i2",

\POS (10,-8) *\cir<0pt>{} ="a3" ,
\POS (20,-8) *\cir<0pt>{} ="b3" ,
\POS (50,-8) *\cir<0pt>{} ="h3",
\POS (60,-8) *\cir<0pt>{} ="i3",

\POS "a" \ar@{-}^<<<{}_<<{}  "b",
\POS "b" \ar@{-}^<<<{}_<<{}  "e",
\POS "e" \ar@{.}^<<<{}_<<{}  "f",
\POS "f" \ar@{-}^<<<{}_<<{} "h",
\POS "h" \ar@{-}^<{}_<<<{} "i",
\POS "i" \ar@{-}^<<<{}_<<{} "j",

\POS "a1" \ar@{-}^<<<{}_<<{}  "b1",
\POS "b1" \ar@{-}^<<<{}_<<{}  "e1",
\POS "e1" \ar@{.}^<<<{}_<<{}  "f1",
\POS "f1" \ar@{-}^<<<{}_<<{} "h1",
\POS "h1" \ar@{-}^<{}_<<<{} "i1",
\POS "i1" \ar@{-}^<<<{}_<<{} "j",

\POS "a2" \ar@{<->} "a3",
\POS "b2" \ar@{<->} "b3",
\POS "h2" \ar@{<->} "h3",
\POS "i2" \ar@{<->} "i3",
\endxy
\end{equation}
(see the table in p.119 of \cite{sa71}). 
In the diagram above, the arrow indicates the non-trivial Galois action.
Furthermore, the black vertex indicates a root in the set of simple roots of a fixed minimal $F$-Levi subgroup $\tbM'_0$ of $\tbG'.$ 
So, we remove only a subset (denoted by $\vartheta$) of pairs ($\Gal(E/F)$-orbits) of white vertices to obtain an $F$-Levi subgroup $\tbM'$ (see \cite[Section 2.2]{sa71} and \cite[Section I.3]{bo79}).
As discussed in Section \ref{basic notation}, the $F$-Levi subgroup $\tbM',$ corresponding to $\Theta = \Delta \setminus \vartheta,$  
is the centralizer in $\tbG'$ of the split component $\bA_{\tbM'}=(\cap_{\alpha \in \Theta} \ker \alpha)^{\circ}.$ 
Then, $\tbM'$ is of the form
\begin{equation} \label{form of tbM'}
{\Res}_{E/F} {\GL}_{n'_1} \times \cdots \times {\Res}_{E/F} {\GL}_{n'_{k'}} \times {\tbG'}_{m'},
\end{equation}
where $\sum_{i=1}^{k'} 2n'_i + m' = n$ with $n'_i \geq 0$ and $m' \geq 2.$ Notice here that $m'$ is always even.
Considering the forms \eqref{form of tbM} and \eqref{form of tbM'}, it is obvious that, if $n_i=n'_i$ for all $i,$ $k=k',$ and $m=m',$ then $\tbM'$ is an $F$-inner form of $\tbM.$
In this case, furthermore, 
two split components $\bA_{\tbM}$ and $\bA_{\tbM'} $ are isomorphic over $F.$ 

\begin{rem}
Since there is a bijection between $H^1(F, (\tbG_n)_{\ad})$ and the set of isomorphism classes of $F$-inner forms of $\tbG_n$ (see Section \ref{basic notation}), using the fact (\cite[Lemma 1.2.1(ii) and p.657]{hala04}) that
\[ 
H^1(F, (\tbG_n)_{\ad}) = \left\{ \begin{array}{ll}
         1, & \mbox{if $n$ is odd},\\
        \ZZ/2\ZZ, & \mbox{if $n$ is even},\end{array} \right .
\]
it follows that, when $n$ is odd, there is no non quasi-split $F$-inner form of $\tbG_n,$ and when  $n$ is even, there is a unique non quasi-split $F$-inner form, $\tbG'_n,$  of $\tbG_n$ up to $F$-isomorphism, as we discussed above.
\end{rem}
\subsection{$\SU_n$ and its inner forms} \label{structure for SUn}
We let $\bG=\bG_n$ denote the quasi-split special unitary group $\SU_n$ with respect to $E/F$ and $J_n.$ Thus,
\[
\bG=\tbG \cap {\Res}_{E/F} {\SL}_n =
\{
g \in {\Res}_{E/F} {\SL}_n : gJ_n {^t\bar{g}} = J_n 
\}.
\]
We denote by $\bM$ an $F$-Levi subgroup of $\bG.$ 
Then, $\bM$ is of the form $\tbM \cap \bG$ and $\bA_{\bM} $ is of the form $\bA_{\tbM} \cap \bG.$ 
{

Due to \eqref{type of AtM}, the $F$-points $A_{\bM}$ of the split component $\bA_{\bM}$ thus equals
\begin{equation*} \label{type of AM}
\{
{\diag} (x_1I_{n_1}, \cdots, x_k I_{n_k}, I_m, {x}_k^{-1}I_{n_k}, \cdots, {x}_1^{-1}I_{n_1}) : x_i \in F^{\times}\}.
\end{equation*} 
}
Especially, when $m \geq 2,$ from \eqref{form of tbM} and \cite[Lemma 2.8]{go06} we have a useful isomorphism
\begin{equation} \label{form of bM}
M \s ({\GL}_{n_1}(E) \times \cdots \times {\GL}_{n_k}(E)) \rtimes G_m,
\end{equation}
where $\sum_{i=1}^k 2n_i + m = n$ with $n_i \geq 0.$
More precisely, for $(g, h)$ with $g \in {\GL}_{n_1}(E) \times \cdots \times {\GL}_{n_k}(E)$ and $h \in G_m,$ the isomorphism \eqref{form of bM} from $({\GL}_{n_1}(E) \times \cdots \times {\GL}_{n_k}(E)) \rtimes G_m$ to $M$ is given by
\begin{equation} \label{imp iso}
(g, h) \longmapsto 
\left( \begin{array}{ccc}
 g& & \\
 & \alpha_m(g)^{-1}h & \\
 & & \varepsilon(g) \\ 
 \end{array} \right) \in \tM \cap {\SU}_n(F) = M,
\end{equation}
where $\alpha_m(g) = \alpha_m(\det g)$ and for $a \in E^{\times},$ 
\begin{equation} \label{def of alpha}
\alpha_m(a) := 
\left( \begin{array}{ccc}
 a& & \\
 & I_{m-2} & \\
 & & \bar{a}^{-1}\\ 
 \end{array} \right) \in {\GL}_{m}(E).
\end{equation}
By convention, we note that ${\bG}_0={\bG}_1 = 1.$
Given $(g, h)$ with $g \in {\GL}_{n_1}(E) \times \cdots \times {\GL}_{n_k}(E)$ and $h \in G_m,$ the action of  $g$ on $h$ is given by
\[
g \circ h = \alpha_m(g) \cdot h \cdot \alpha_m(g)^{-1}.
\]

Next, we let $\bG'=\bG'_n$ denote an $F$--inner form of $\bG.$
As discussed in Section \ref{structure for Un},
for $n$ \textbf{odd}, there is \textit{no non quasi-split} $F$-inner form of $\bG_n.$
On the other hand, for $n$ \textbf{even},
there is a \textit{unique} non quasi-split $F$-inner form $\bG'_n,$ up to $F$-isomorphism, whose $\Gamma$-diagram is \eqref{diagram}.
Any $F$-Levi subgroup $\bM'$ of $\bG'$ is of the form  $\tbM' \cap \bG'.$ 
Thus, from \eqref{form of tbM'} and \eqref{form of bM}, we have an isomorphism
\begin{equation} \label{form of bM'}
M' \s 
({\GL}_{n'_1}(E) \times \cdots \times {\GL}_{n'_{k'}}(E)) \rtimes {G'}_{m'},
\end{equation}
where $\sum_{i=1}^{k'} 2n'_i + m' = n$ with $n'_i \geq  0$ and $m' \geq 2.$ Notice as before that $m'$ is always even.
To explain \eqref{form of bM'} more precisely, 
for $(g', h')$ with $g' \in {\GL}_{n'_1}(E) \times \cdots \times {\GL}_{n'_k}(E)$ and $h' \in {G'}_{m'},$ the isomorphism \eqref{form of bM'} from $({\GL}_{n'_1}(E) \times \cdots \times {\GL}_{n'_{k'}}(F)) \rtimes {G'}_{m'}$ to $M'$ is given by
\[
(g', h') \longmapsto 
\left( \begin{array}{ccc}
 g'& & \\
 & \alpha_{m'}(g')^{-1}h' & \\
 & & \varepsilon(g') \\ 
 \end{array} \right) \in \tM' \cap {G'}_n = M',
\]
where $\alpha_{m'}(g') = \alpha_{m'}(\det g')$ and for $a \in E^{\times},$ 
\[
\alpha_{m'}(a) := 
\left( \begin{array}{ccc}
 a& & \\
 & I_{m'-2} & \\
 & & \bar{a}^{-1}\\ 
 \end{array} \right) \in {\GL}_{m'}(E).
\]
By convention, we note that, ${\bG'}_0 = 1.$
Given $(g', h') \in M'$ with  $g' \in {\GL}_{n'_1}(E) \times \cdots \times {\GL}_{n'_k}(E)$ and $h' \in G'_{m'},$ the action of $g'$ on $h'$ is given by
\[
g' \circ h' = \alpha_{m'}(g') \cdot h' \cdot \alpha_{m'}(g')^{-1}.
\]
Like the unitary case in Section \ref{structure for Un}, it is obvious from \eqref{form of bM} and \eqref{form of bM'} that, if $n_i=n'_i$ for all $i,$ $k=k',$ and $m=m',$ then $\bM'$ is an $F$-inner form of $\bM.$ 
In this case, 
there is an $F$--isomorphism $\bA_{\bM} \s \bA_{\bM'}$ between two split components.

We have following exact sequences of algebraic groups
\[
1 \longrightarrow {\bG} \longrightarrow {\tbG}  \overset{\det}{\longrightarrow}  {\U}_1 \longrightarrow 1,
\]
and 
\[
1 \longrightarrow {\bG'} \longrightarrow {\tbG'} \overset{\det}{\longrightarrow}  {\U}_1 \longrightarrow 1.
\]
Applying Galois cohomology, since $H^1(F, \bG)=H^1(F, \bG')=1$  due to \cite[Theorem 6.4]{pr94}, we have following exact sequences of their $F$-points
\[
1 \longrightarrow {G} \longrightarrow {\tG} \overset{\det}{\longrightarrow}  E^{1} \longrightarrow 1,
\]
and 
\[
1 \longrightarrow {G'} \longrightarrow {\tG'} \overset{\det}{\longrightarrow}  E^{1} \longrightarrow 1,
\]
where $E^1 = \{ x \in E : N_{E/K} (x) = x \bar x = 1 \}.$
All above exact sequences are true for $F$-Levi subgroups. 
In particular, $H^1(F, \bM) =H^1(F, \bM')= 1$ since $H^1(F, \bM) \hookrightarrow H^1(F, \bG)=1$ and $H^1(F, \bM') \hookrightarrow H^1(F, \bG')=1$
(see \cite[p.95, footnote]{morel}, \cite[p.270]{kot97}, and \cite[Remark 2.5]{chgo12}).
We further have
\[
{\tbG}_{\der} = {\bG}_{\der}=\bG, ~~{\tbG'}_{\der} = {\bG'}_{\der}=\bG', ~~{\tbM}_{\der} = {\bM}_{\der} \subset \bM, ~~{\tbM'}_{\der} = {\bM'}_{\der} \subset  \bM'.
\]

\subsection{$L$-groups} \label{L-groups}
We describe $L$-groups of $\tbG=\U_n, \bG=\SU_n,$ their inner forms $\tbG', \bG',$ and the $L$--groups of  their $F$-Levi subgroups $\tbM, \tbM', \bM,$ and  $\bM'.$ 
Furthermore, we explain a relationship between $L$-groups of $\tbM$ and $\bM,$ and investigate a connection between two $\Gamma$-split components ${\bA}_{\widehat{\tM}}$ and ${\bA}_{\widehat{M}} $ of $\widehat{\tM}$ and $\widehat{M},$ defined as follows
\[
{\bA}_{\widehat{\tM}} := (Z(\widehat{\tM})^{\Gamma})^\circ \quad \text{and} \quad {\bA}_{\widehat{M}} := (Z(\widehat{M})^{\Gamma})^\circ.
\]
These will be used in Section \ref{R-groups for SU(n) and its inner forms}.
Based on \cite{mok13}, we set
\[
{^L}{\tG}={^L}{\tG'}={\GL}_n(\CC)\rtimes W_F,
\]
where $W_E$ acts trivially on ${\GL}_n(\CC),$ and the action of $w_c \in W_F \setminus W_E$ on $ \hat g \in \GL_n(\CC)$ is given by
\begin{equation} \label{galois action on L-gp}
w_c( \hat g) = J_n^{-1} {^t}\hat g^{-1} J_n. 
\end{equation}
For $\bG$ and $\bG',$ the $L$-group is
\[
{^L}{G}={^L}{G'}={\PGL}_n(\CC)\rtimes W_F,
\]
where $W_E$ acts trivially on ${\PGL}_n(\CC),$ and the action of $w_c \in W_F \setminus W_E$ on $\hat g \in \PGL_n(\CC)$ is given by \eqref{galois action on L-gp}.

Let $\tbM,$ $\tbM',$ $\bM,$ and $\bM'$ be $F$--Levi subgroups of $\tbG,$ $\tbG',$ $\bG,$ and $\bG',$ respectively, such that $\tbM'$ is an $F$--inner form of $\tbM,$ and $\bM'$ is an $F$--inner form of $\bM.$
We have
\begin{equation*}  \label{iso between L-groups}
\widehat \tM = \widehat{\tM'} \subset \widehat \tG = \widehat{\tG'}, \quad {^L}{\tM}={^L}{\tM'}=\widehat \tM \rtimes W_F, \quad \widehat M = \widehat{M'}\subset\widehat G = \widehat{G'}, \quad {^L}{M}={^L}{M'}=\widehat M\rtimes W_F.
\end{equation*}
Considering \eqref{form of tbM} and \eqref{form of tbM'}, we set
\[
\tbM = {\Res}_{E/F} {\GL}_{n_1} \times \cdots \times {\Res}_{E/F} {\GL}_{n_k} \times \tbG_m,
\]
\[
\tbM' = {\Res}_{E/F} {\GL}_{n_1} \times \cdots \times {\Res}_{E/F} {\GL}_{n_k} \times \tbG'_m,
\]
where $\tbG'_m$ is an $F$-inner form of $\tbG_m,$ $\sum_{i=1}^k 2n_i + m = n$ with $n_i \geq  0,$ and $m \geq 0.$
Then, we have
\[
\widehat{\tbM}=\widehat{\tbM'} = ({\GL}_{n_1}(\CC)\times{\GL}_{n_1}(\CC)) \times \cdots \times ({\GL}_{n_k}(\CC)\times{\GL}_{n_k}(\CC))\times {\GL}_m(\CC).
\]
$W_E$ acts trivially on $\widehat{\tbM},$ and the action of $w_c \in W_F \setminus W_E$ on $((\hat g_{11}, \hat g_{12}), (\hat g_{21}, \hat g_{22}), \cdots, (\hat g_{k1}, \hat g_{k2}), \hat h) \in \widehat{\tbM}$ is given by 
\begin{equation} \label{galois action on L-gp of tM}
w_c((\hat g_{11}, \hat g_{12}), (\hat g_{21}, \hat g_{22}), \cdots, (\hat g_{k1}, \hat g_{k2}), \hat h)
=
((\hat g_{12}, \hat g_{11}), (\hat g_{22}, \hat g_{21}), \cdots, (\hat g_{k2}, \hat g_{k1}), 
 J_m^{-1} {^t}\hat h^{-1} J_m). 
\end{equation}

Next, we note that, since $\tbM_{\der}$ and $\bM_{\der}$ are simply connected, we have
\[
\widehat{\tM_{\der}} =
 {(\widehat{\tM})}_{\ad} =
  \widehat{\tM}/Z(\widehat{\tM}), ~~
   \widehat{M_{\der}} = {\widehat{(M)}}_{\ad} = \widehat{M}/Z(\widehat{M}). 
\] 
Further, from \cite[(1.8.1) p.616]{kot84}, the exact sequence of algebraic groups
\[
1 \longrightarrow \bM \longrightarrow \tbM \longrightarrow {\U}_1 \longrightarrow 1
\]
yields an exact sequence
\begin{equation} \label{exact of L centers}
1 \longrightarrow \widehat{{\U}_1}=\CC^{\times} \longrightarrow Z({\widehat{\tM}}) \s (\CC^{\times})^{2k+1} \longrightarrow Z({\widehat{M}}) \longrightarrow 1.
\end{equation}
We thus have the following commutative diagram of $L$-groups (cf., \cite[Remark 2.4]{chgo12})
\[
\begin{CD}
@.  1 @. 1 @.  @.
\\
@.      @VVV        @VVV   @. @.\\
1 @>>> \widehat{{\U}_1} = \CC^{\times} @>{\s}>> \ker @>>> 1 @. @. 
\\
@.      @VVV        @VVV   @VVV  @.\\
1 @>>> Z(\widehat{\tM}) = (\CC^{\times})^{2k+1} @>>> \widehat{\tM} @>>> \widehat{\tM_{\der}} @>>> 1 
\\
@.      @VVV          @VVV   @|  @.
\\
1 @>>> Z(\widehat{M}) = (\CC^{\times})^{2k+1} / \CC^{\times} @>>> \widehat{M} @>>> \widehat{M_{\der}} @>>> 1 \\
@.      @VVV          @VVV   @VVV  @.\\
@.  1 @. 1 @. 1 @. 
\end{CD}
\]
The middle vertical exact sequence becomes
\begin{equation} \label{exact of L-gp of M tM}
1 \longrightarrow \widehat{{\U}_1}=\CC^{\times} \longrightarrow \widehat{\tM}=\widehat{\tM'} \longrightarrow \widehat{M}=\widehat{M'} \longrightarrow 1.
\end{equation}
We also have
\begin{equation} \label{exact of L-gp of G tG}
1 \longrightarrow \widehat{{\U}_1}=\CC^{\times} \longrightarrow \widehat{\tG}=\widehat{\tG'} \longrightarrow  \widehat{G}=\widehat{G'} \longrightarrow 1.
\end{equation}

Note that $\widehat{\U_1}$ in \eqref{exact of L-gp of M tM} and \eqref{exact of L-gp of G tG} equals $ Z(\widehat{\tG})=Z(\widehat{\tG'}).$
Furthermore, $\widehat{\U_1}$ is diagonally embedded into $\widehat{\tG}$ in \eqref{exact of L-gp of G tG} and the action of $W_F$ on $\widehat{\U_1} = \CC^{\times}$ is obtained from \eqref{galois action on L-gp}. 
The action of $W_F$ can be also obtained from \eqref{galois action on L-gp of tM}, since it is diagonally embedded into  $Z(\widehat{\tM})$ in \eqref{exact of L centers} as well.
It thus follows that the subgroups of $\Gamma=\Gal(\bar F/F)$-invariants satisfy
\begin{equation} \label{gamma invariants}
Z(\widehat{\tG})^{\Gamma}=Z(\widehat{\tG'})^{\Gamma} = \widehat{{\U}_1}^{\Gamma} = \{ \pm 1 \}.
\end{equation}
Moreover, using the action of $W_F$ on $\widehat{\tM}=\widehat{\tM'}$ in \eqref{galois action on L-gp of tM} and the surjective map $\widehat{\tM}=\widehat{\tM'} \twoheadrightarrow \widehat{M}=\widehat{M'}$ in \eqref{exact of L-gp of M tM}, the action of $W_F$ on $\widehat M = \widehat{M'}$ can be obtained. 

From \eqref{galois action on L-gp of tM}, we have
\[
{\bA}_{\widehat{\tM}} := (Z(\widehat{\tM})^{\Gamma})^\circ \s ((\CC^{\times})^k \times \{ \pm 1 \})^\circ = (\CC^{\times})^k.
\]
We note that
\[
{\bA}_{\widehat{\tM}}/Z(\widehat{\tG})^{\Gamma} \s ({\bA}_{\widehat{\tM}} \cdot Z(\widehat{\tG}))/Z(\widehat{\tG}) \subset \widehat{G}.
\]
Then, we have the following lemmas.
\begin{lm} \label{before gamma-inv of centers}
The quotient ${\bA}_{\widehat{\tM}}/Z(\widehat{\tG})^{\Gamma}$ is connected.
\end{lm}
\begin{proof}
Since ${\bA}_{\widehat{\tM}} = (Z(\widehat{\tM})^{\Gamma})^\circ$ by definition and since $Z(\widehat{\tG}) \subset {\bA}_{\widehat{\tM}},$ we have
\[
{\bA}_{\widehat{\tM}}/Z(\widehat{\tG})^{\Gamma}  = (Z(\widehat{\tM})^{\Gamma})^\circ / (Z(\widehat{\tG})^{\Gamma} \cap (Z(\widehat{\tM})^{\Gamma})^\circ)
\]
which is isomorphic to $Z(\widehat{\tM})^{\Gamma}/Z(\widehat{\tG})^{\Gamma} $ by \cite[Lemma 1.1]{art99}.
Note that $Z(\widehat{\tM})^{\Gamma}/Z(\widehat{\tG})^{\Gamma} $ is connected due to the proof of \cite[Lemma 1.1]{art99}. 
\end{proof}
From \eqref{exact of L centers} and \eqref{gamma invariants}, we have the following exact sequence
\[
1 \longrightarrow  {\bA}_{\widehat{\tM}}/Z(\widehat{\tG})^{\Gamma} \longrightarrow Z(\widehat{M})^{\Gamma} \longrightarrow  H^1(F, Z(\widehat{\tG})). 
\]
The following lemma proves that the embedding ${\bA}_{\widehat{\tM}}/Z(\widehat{\tG})^{\Gamma}  \hookrightarrow Z(\widehat{M})^{\Gamma}$ is in fact an equality.
\begin{lm} \label{gamma-inv of centers}
With the notation above, we have 
\[
{\bA}_{\widehat M} = Z(\widehat{M})^{\Gamma}  ~~\text{and}~~
{\bA}_{\widehat{M}} = {\bA}_{\widehat{\tM}}/Z(\widehat{\tG})^{\Gamma}.
\]
\end{lm}
\begin{proof}
We note from \cite[Lemma 1.1]{art99} that
\[
Z(\widehat{M})^{\Gamma} = Z(\widehat{G})^{\Gamma} \cdot (Z(\widehat{M})^{\Gamma})^{\circ}.
\]
Since $Z(\widehat{G}) = 1$ and $A_{\widehat M} = (Z(\widehat{M})^{\Gamma})^{\circ},$ the first equality is verified (hence, $Z(\widehat{M})^{\Gamma}$ is connected).
We note that
\begin{equation*} \label{an inclusion}
{\bA}_{\widehat{\tM}}/Z(\widehat{\tG})^{\Gamma} \subset Z(\widehat{M})^{\Gamma} = {\bA}_{\widehat{M}}
\end{equation*}
and ${\bA}_{\widehat{\tM}}/Z(\widehat{\tG})^{\Gamma}$ is connected by Lemma \ref{before gamma-inv of centers}.
Since ${\bA}_{\widehat{M}}$ is a maximal torus in $C_{\phi}^{\circ},$ and since ${\bA}_{\widehat{\tM}}/Z(\widehat{\tG})^{\Gamma}$ is also a torus having the same dimension with ${\bA}_{\widehat{M}}$ (cf., Lemma \ref{lm for identity comp} in Section \ref{arthur conj}), 
we have ${\bA}_{\widehat{\tM}}/Z(\widehat{\tG})^{\Gamma} = {\bA}_{\widehat{M}}.$
\end{proof}
\section{Weyl group actions} \label{weyl group actions}
For an $F$-Levi subgroup $\bM$ of a connected reductive algebraic group $\bG,$ we recall from \ref{basic notation} that
the Weyl groups are 
$W_M = W(\bG, \bA_\bM) := N_\bG(\bA_\bM) / Z_\bG(\bA_\bM),$
$W_{M'} = W(\bG', \bA_{\bM'}) := N_{\bG'}(\bA_{\bM'}) / Z_{\bG'}(\bA_{\bM'}),$ and
$W_{\widehat M} = W(\widehat G, A_{\widehat M}) := N_{\widehat G}(A_{\widehat M}) / Z_{\widehat G}(A_{\widehat M}).$
Through the duality
\begin{equation*} \label{duality}
s_{\alpha} \mapsto s_{\alpha^{\vee}}
\end{equation*}
between simple reflections for $\alpha \in \Delta,$ we have 
\begin{equation*}  \label{two isos}
W_M \s W_{\widehat M} \quad W_{\widehat{M}'} \s W_{M'}.
\end{equation*} 
We thus identify
\begin{equation}  \label{identity for Weyls}
W_M = W_{\widehat M}=W_{\widehat{M}'} =W_{M'}.
\end{equation}
\subsection{On Levi subgroups} \label{WM-actions on Levis}
For simplicity, in Sections \ref{WM-actions on Levis} and \ref{WM-actions on rep of Levis}, we will write $\tbG$ for \textit{both} quasi-split unitary groups $\U_n$ and its non quasi-split $F$--inner forms $\tbG',$ and $\bG$ for \textit{both} quasi-split special unitary groups $\SU_n$ and its non quasi-split $F$--inner forms $\bG'.$

Let $\tbM$ and $\bM$ be $F$-Levi subgroups of $\tbG$ and $\bG,$ respectively. We recall from Section \ref{structure} that
\[
\tM \s
{\GL}_{n_1}(E) \times \cdots \times  {\GL}_{n_{k}}(E) \times \tG_m,
\]
\[
M \s
({\GL}_{n_1}(E) \times \cdots \times  {\GL}_{n_{k}}(E)) \rtimes G_m,
\]
where $\sum_{i=1}^{k} 2n_i + m = n$ with $n_i \geq  0$ and $m \geq 0.$ Notice here that $m\geq 2$ and is always even for non quasi-split inner forms.
We describe the action of Weyl group on Levi subgroups $\tbM$ and $\bM$ as well as irreducible representations of $\tM$ and $M,$ based on the results in \cite{go95, go06}.
We denote by $S_k$ the symmetric group in $k$ letters.
From \cite{go95, go06}, we have
\begin{equation} \label{desc WM}
W_{\tM} = W_{M} = \subset S_k \ltimes \ZZ_2^k.
\end{equation}

More precisely, $W_{\tM}\simeq S\ltimes\sC,$ 
where  $S=\langle (ij)|n_i=n_j \rangle,$ and $\sC = \ZZ_2^k.$ 
For $g \in \tM,$ write
\[
g=(g_1,\dots g_i,\dots,g_j,\dots, g_k,h) \in GL_{n_1}(E) \times GL_{n_2}(E) \times \cdots \times GL_{n_k}(E) \times \tG_m.
\]
The permutation $(ij)$ acts on $g \in \tM$ by
\begin{equation}  \label{action ij}
(ij):(g_1,\dots g_i,\dots,g_j,\dots, g_k,h) \mapsto (g_1,\dots,g_j,\dots,g_i,\dots,g_k,h).
\end{equation}
The finite $2$-group $\ZZ_2^k$ is generated by ``block sign changes" $C_i$ which acts on $g \in \tM$ by
\begin{equation}  \label{action Ci-s}
C_i:(g_1,\dots,g_i,\dots,g_k,h) \mapsto (g_1,\dots,\, \varepsilon(g_i),\dots,g_k,h),
\end{equation}
where $\varepsilon(g_i) = {^t}\bar g_i^{-1}.$

\subsection{On representations of Levi subgroups} \label{WM-actions on rep of Levis}
Set $\tsigma$ to be $\tsigma_1 \otimes \tsigma_2 \otimes \cdots \otimes \tsigma_k \otimes \ttau.$
From \eqref{action ij} and \eqref{action Ci-s}, we have
\begin{gather*}
(ij)\tsigma=\tsigma_1\otimes\cdots
\otimes\tsigma_j\otimes\cdots
\otimes\tsigma_i\otimes\cdots
\otimes\tsigma_k\otimes\ttau;\\
C_i\tsigma=\tsigma_1\otimes\cdots
\otimes\varepsilon(\tsigma_i)\otimes\cdots
\otimes\tsigma_k\otimes\ttau,
\end{gather*}
where $\varepsilon(\tsigma_i)(g_i)=\tsigma_i(\varepsilon(g_i)),$
and these describe the action of $W_{\tM}$ on $\tsigma.$

Now we turn to the case of $\bM.$ 
Note that, since $G_0=G_1=1,$ $W_M$ acts on $M$ and an irreducible representation of $M$ in the same manner for $m=0, 1.$ 
In particular,
when $m=0$ (thus, $n$ is even), 
from the proof of \cite[Lemma 2.4]{go06},
$M$ is of the form
\[
M = \{
\left( \begin{array}{cc}
 g & \\
 & \varepsilon(g) \\ 
 \end{array} \right) : \det (g) \det(\varepsilon(g)) = 1, ~ g \in GL_{n_1}(E) \times GL_{n_2}(E) \times \cdots \times GL_{n_k}(E) \},
\] 
which implies $M \s \{ g \in GL_{n_1}(E) \times GL_{n_2}(E) \times \cdots \times GL_{n_k}(E) : \det g \in F^{\times} \}.$  
When $m=1$ (thus, $n$ is odd), since $\tM \s GL_{n_1}(E) \times GL_{n_2}(E) \times \cdots \times GL_{n_k}(E) \times {\U}_1(F),$ we have
\[
M \s GL_{n_1}(E) \times GL_{n_2}(E) \times \cdots \times GL_{n_k}(E).
\]
Note $\U_1(F) = E^{1} = \{ x \in E : x \bar x=1 \}.$
Further, for any $\tsigma \s \tsigma_1 \otimes \tsigma_2 \otimes \cdots \otimes \tsigma_k \otimes \omega \in \Irr(\tM),$ we have
\[
{\Res}_{G_n}^{\tG_n} \s \tsigma_1 \otimes \tsigma_2 \otimes \cdots \otimes \tsigma_k \otimes \omega \omega^ \varepsilon,
\]
which is always irreducible
(see the proof of \cite[Lemma 2.5]{go06}).
Let us move to the case of $m\geq 2.$
For $g \in M,$ from \eqref{form of bM'}, we write
\[
g=(g_0, h)=(g_1,\dots g_i,\dots,g_j,\dots, g_k,h) \in (GL_{n_1}(E) \times GL_{n_2}(E) \times \cdots \times GL_{n_k}(E)) \rtimes G_m.
\]
Denote $wgw^{-1}$ by $(wg_0, h)$ for $w \in W_{\tM}.$
Following arguments in \cite[p.353]{go06}, 
the permutation $(ij)$ acts on $g \in M$
\begin{align} \label{action ij-s}
(ij):(g_1,\dots g_i,\dots,g_j,\dots, g_k,h) \mapsto 
& (g_1,\dots,g_j,\dots,g_i,\dots,g_k, 
\alpha_m(g_0)^{-1}\alpha_m((ij)g_0)h) \\
&= (g_1,\dots,g_j,\dots,g_i,\dots,g_k, \nonumber
h),
\end{align} 
since $\alpha_m(g_0)^{-1}\alpha_m((ij)g_0) = I_m.$
That is, the permutation $(ij)$ acts trivially on $G_m.$

The finite $2$-group $\ZZ_2^k$ is generated by ``block sign changes" $C_i$ which acts on $g \in M$
\begin{equation}  \label{action Ci}
C_i:(g_1,\dots,g_i,\dots,g_k,h) \mapsto (g_1,\dots,\, \varepsilon(g_i),\dots,g_k,\alpha_m(g_0)^{-1}\alpha_m(C_ig_0)h).
\end{equation}
Note from the definition of $\alpha_m$ in \eqref{def of alpha} that 
\[
\alpha_m(g_0)^{-1}\alpha_m(C_ig_0) := 
\left( \begin{array}{ccc}
 (\det(g_i)\det(\bar g_i))^{-1}& & \\
 & I_{m-2} & \\
 & & \det(g_i)\det(\bar g_i)\\ 
 \end{array} \right) \in {\SU}_{m}(F).
\]

Given $\sigma \in \Irr(M),$ fix a lift $\tsigma \in \Irr(\tM).$
Set $\tsigma \in \Irr(\tM)$ to be $\sigma_1 \otimes \sigma_2 \otimes \cdots \otimes \sigma_k \otimes \ttau,$ 
and $\ttau$ to be a lift in $\Irr(\tG_m)$ with $\tau \hookrightarrow \Res^{\tG_m}_{G_m} {\ttau}.$
Write $V_{\sigma}, V_{\tsigma_i},$ for $i=1, \dots, k,$ $V_{\ttau},$ and $V_\tau$ for the corresponding representation $\CC$-vector spaces.
Then, $V_\sigma$ is of the form
\[
V_{\tsigma_1} \otimes V_{\tsigma_2} \otimes\cdots \otimes V_{\tsigma_k} \otimes V_\tau.
\]
From \cite[p. 353]{go06}, for $(g_0, h)=(g_1,\dots g_i,\dots,g_j,\dots, g_k,h) \in (GL_{n_1}(E) \times GL_{n_2}(E) \times \cdots \times GL_{n_k}(E)) \rtimes G_m,$ the representation $\sigma$ acts on $(v_1, v_2, \cdots, v_k, v_0) \in V_\sigma$ as follows,
\begin{equation*} \label{any sigma equ}
\sigma(g_0,h)(v_1, v_2, \cdots, v_k, v_0) =  \sigma_1(g_1)(v_1) \otimes \sigma_2(g_2)(v_2) \otimes \cdots \otimes \sigma_k(g_k)(v_k) \otimes \ttau(\alpha_m(g_0)^{-1})\tau(h)(v_0).
\end{equation*}
From \eqref{action ij-s} and \cite[p. 353]{go06}, 
we have
\begin{align*} \label{action ij on rep}
(ij)\sigma(g_0,h) =
& \sigma_1(g_1) \otimes \cdots \otimes \sigma_j(g_i) \otimes \cdots \otimes \sigma_i(g_j)\cdots \otimes\sigma_k(g_k) \otimes \ttau(\alpha_m((ij)g_0)^{-1})\tau(\alpha_m(g_0)^{-1}\alpha_m((ij)g_0)h) \\
= & \sigma_1(g_1) \otimes \cdots \otimes \sigma_i(g_j) \otimes \cdots \otimes \sigma_j(g_i)\cdots \otimes\sigma_k(g_k) \otimes \ttau(\alpha_m(g_0)^{-1})\tau(h).
\end{align*}  
Likewise, using \eqref{action Ci},
we have
\begin{align*} 
C_i\sigma(g_0,h) =
& \sigma_1(g_1) \otimes \cdots \otimes \varepsilon(\sigma_i)(g_i) \otimes \cdots \otimes\sigma_k(g_k) \otimes \ttau(\alpha_m(C_ig_0)^{-1})\tau(\alpha_m(g_0)^{-1}\alpha_m(C_ig_0)h) \\
= & \sigma_1(g_1) \otimes \cdots \otimes \sigma_i(\varepsilon(g_i)) \otimes \cdots \otimes \sigma_k(g_k) \otimes \ttau(\alpha_m(g_0)^{-1})\tau(h).
\end{align*}
Therefore, $W_{M}$ acts non-trivially only on $\sigma_1,\sigma_2, \dots, \sigma_k,$ but trivially on $\tau.$
\section{Revisiting $R$--groups for $\U_n$ and their inner forms} \label{R-groups for U(n) and its inner forms}
Based on some known results in \cite{bangoldberg12, go95, kmsw14, mok13} regarding $R$-groups for $\U_n$ and its $F$-inner form, we discuss Arthur's conjecture for $\U_n$ and its inner forms, behavior of $R$-groups within $L$-packets of $\U_n$ and its inner forms, and behavior of $R$-groups between $\U_n$ and its inner forms.
\subsection{$L$-packets for $\U_n$ and its inner forms} \label{LLC for unitary}
Let $\tbG=\tbG_n$ denote the quasi-split unitary group $\U_n$ with respect to $E/F$ and $J_n,$ and $\tbM$ an $F$-Levi subgroup of $\tbG.$
For our purpose of studying $R$-groups, we focus on $\Phi_{\temp}(\tG).$
In \cite[Theorem 2.5.1.(b)]{mok13}, Mok generalized Rogawski's results (\cite{rog90}) in the case of unitary groups in three variables as follows.
There is a surjective finite-to-one map
\[
\Pi_{\temp}(\tG) \longrightarrow \Phi_{\temp}(\tG),
\]
and for $\tphi \in \Phi_{\temp} (\tG),$ the tempered $L$-packet $\Pi_{\tphi}(\tG)$ is constructed.
The same is true for an $F$-inner form $\tbG'$ of $\tbG$ by Kaletha-Minguez-Shin-White \cite[Section 1.6.1]{kmsw14}.

Let $\tbM$ and $\tbM'$ be $F$-Levi subgroups of $\tbG$ and $\tbG',$ respectively, which are $F$-inner forms of each other. Then, from \eqref{form of tbM} and \eqref{form of tbM'}, we recall
\[
\tbM \s {\Res}_{E/F} {\GL}_{n_1} \times \cdots \times {\Res}_{E/F} {\GL}_{n_k} \times \tbG_m,
\]
\[
\tbM' \s {\Res}_{E/F} {\GL}_{n_1} \times \cdots \times {\Res}_{E/F} {\GL}_{n_k} \times \tbG'_m,
\]
where $\tbG'_m$ is an $F$-inner form of $\tbG_m,$ $\sum_{i=1}^k 2n_i + m = n$ with $n_i \geq  0,$ and $m \geq 0.$
Let $\tphi \in \Phi_{\disc}(\tM)$ be given. 
Then, $\tphi$ is $\tbM$-relevant (see \cite[Section 8.2]{bo79}) and lies in $\Phi_{\disc}(\tM')$ as well.
We note that $\tphi$ is of the form $\tphi_1 \oplus \tphi_2 \oplus \cdots \oplus \tphi_k \oplus \tphi_-,$ 
where $\tphi_i \in \Phi_{\disc}(\GL_{n_i}(E))$ 
and $\tphi_- \in \Phi_{\disc}(\tG_m) = \Phi_{\disc}(\tG'_m).$
For each $\tphi_i,$ due to \cite{ht01, he00, scholze13}, we construct $L$-packets $\Pi_{\tphi_i}(\GL_{n_i}(E))$ consisting of discrete series representations of $\GL_{n_i}(E).$ 
Note that $\Pi_{\tphi_i}(\GL_{n_i}(E))$ is a singleton.
For $\tphi_-,$ due to \cite{kmsw14, mok13, rog90},
we construct $L$-packets $\Pi_{\tphi_-}(G_m)$ and $\Pi_{\tphi_-}(G'_m)$ consisting of discrete series representations of $G_m$ and $G'_m,$ respectively.
By taking the tensor product of members in packets for each $\tphi_i$ and $\tphi_-,$ we thus construct $L$-packets ${\Pi}_{\tphi}(\tM)$ of $\tM$ and ${\Pi}_{\tphi}(\tM')$ of $\tM',$ associated to the elliptic tempered $L$-parameter $\tphi.$
\subsection{Invariance of $R$-groups between $\U_n$ and its inner forms} \label{invariance for unitary}
\begin{thm} (Goldberg, \cite[Theorem 3.4]{go95}) \label{r-groups for tG}
Let $\tbM$ be an $F$-Levi subgroup of $\tbG.$
Given $\tsigma \s \tsigma_1 \otimes \tsigma_2 \otimes \cdots \otimes \tsigma_k \otimes \ttau \in \Pi_{\disc}(\tM),$ we have
\[
R_{\tsigma} \s \ZZ_2^d,
\]
where $d$ is the number of inequivalent $\tsigma
_i$ such that the induced representation
$\ii_{\tG_{n_i + m}, ~{\GL}_{n_i}(E)\times \tG_m} (\tsigma \otimes \ttau)$ is reducible. 
\qed
\end{thm}

\begin{thm} \label{iso bw R for tG}
Let $\tphi \in \Phi_{\disc}(\tM)$ be given. 
Under the identity \eqref{identity for Weyls}, for any $\tsigma \in \Pi_{\tphi}(\tM)$ and $\tsigma' \in \Pi_{\tphi}(\tM'),$ we have
\begin{equation*} \label{identity of W}
W(\tsigma) = W_{\tphi} = W(\tsigma').
\end{equation*}
and 
\begin{equation} \label{identity of W'}
W^{\circ}_{\tsigma} = W^{\circ}_{\tphi} = W^{\circ}_{\tsigma'}.
\end{equation}
Therefore, we have
\[
R_{\tsigma} \s R_{\tphi, \tsigma} \s R_{\tphi} \s R_{\tphi, \tsigma'} \s R_{\tsigma'}.
\]
\end{thm}
\begin{proof}
This is a consequence of \cite[Lemma 11]{bangoldberg12} for maximal Levi subgroups and \cite[Section 7.6]{mok13} and \cite[Section 4.6]{kmsw14} for general Levi subgroups.
\end{proof}
We have following corollaries to Theorems \ref{r-groups for tG} and \ref{iso bw R for tG}.
\begin{cor}
Let $\tphi \in \Phi_{\disc}(\tM)$ be given. 
For $\tsigma \in \Pi_{\tphi}(\tM)$ and $\tsigma' \in \Pi_{\tphi}(\tM'),$ we have $\ii_{\tG,\tM}(\tsigma)$ is irreducible if and only if $\ii_{\tG',\tM'}(\tsigma')$ is irreducible.
\qed
\end{cor}
\begin{cor} \label{r-groups for G'=bc}
Let $\tbM'$ be an $F$-Levi subgroup of $\tbG'.$
Given $\tsigma' \s \tsigma'_1 \otimes \tsigma'_2 \otimes \cdots \otimes \tsigma'_k \otimes \ttau' \in \Pi_{\disc}(\tM'),$ we have
\[
R_{\tsigma'} \s \ZZ_2^d,
\]
where $d$ is the number of inequivalent $\tsigma'
_i$ such that the induced representation
$\ii_{\tG'_{n_i + m}, ~{\GL}_{n_i}(E)\times \tG'_m} (\tsigma' \otimes \ttau')$ is reducible. 
\qed
\end{cor}
\begin{cor} \label{trivial on w}
Suppose $w \in R_{\tsigma'}$ and $w=sc,$ with $s\in S$ and $c\in\sC.$  Then $w=1.$
\end{cor}
\begin{proof}
This is a consequence of \cite[Lemma 3.2]{go95} and Theorem \ref{iso bw R for tG}.
\end{proof}

Let $\tsigma' \in \Pi_{\disc}(\tM')$ be given. Denote by $\CC[R(\tsigma')]_{\eta}$ the group algebra of $R(\tsigma')$ twisted by a $2$-cocycle $\eta,$ and by $C(\tsigma'),$ known as the commuting algebra of $\ii_{\tG',\tM'} (\tsigma'),$ the algebra ${\End}_{\tG'}(\ii_{\tG',\tM'} (\tsigma'))$ of $\tG'$-endomorphisms of $\ii_{\tG',\tM'} (\tsigma').$  
Then, we have 
\[
C(\tsigma') \s \CC[R(\tsigma')]_{\eta}
\]
as group algebras (see \cite{ks72, sil78, sil78cor}).
\begin{pro} \label{split 2-cocycle for G' bc}
With the above notation, we have 
\[
C(\tsigma') \s \CC[R_{\tsigma'}].
\]
\end{pro}
\begin{proof}
The proof is similar to \cite[Proposition 4.1]{go95}, due to Theorem \ref{iso bw R for tG} and Corollaries \ref{r-groups for G'=bc} and \ref{trivial on w}.
\end{proof}
As consequence of Proposition \ref{split 2-cocycle for G' bc}, we have the following corollary.
\begin{cor}
Let $\tsigma' \in \Pi_{\disc}(\tM')$ be given. Then, each constituent of $\ii_{\tG',\tM'}(\tsigma')$ appears with multiplicity one.
\qed
\end{cor}

\begin{pro} \label{analogue of go95 thm4.3}
Let $\tsigma' \in \Pi_{\disc}(\tM')$ be given. Then, $\ii_{\tG', \tM'}(\tsigma')$ has an elliptic constituent if and only if all constituents of $\ii_{\tG',\tM'}(\tsigma')$ are elliptic if and only if 
\[
R_{\tsigma'} \s \ZZ_2^k.
\]
\end{pro}
\begin{proof}
The proof is similar to \cite[Theorem 4.3]{go95}, due to Corollary \ref{trivial on w} and Proposition \ref{split 2-cocycle for G' bc}.
\end{proof}
\begin{cor}
Let $\tphi \in \Phi_{\disc}(\tM)$ be given. 
For any $\tsigma \in \Pi_{\tphi}(\tM)$ and $\tsigma' \in \Pi_{\tphi}(\tM'),$  
there is an elliptic constituent in $\ii_{\tG, \tM}(\tsigma)$ 
if and only if 
there is an elliptic constituent in $\ii_{\tG', \tM'}(\tsigma').$
\end{cor}
\begin{proof}
This is a consequence of \cite[Theorem 4.3]{go95} and Proposition \ref{analogue of go95 thm4.3}.
\end{proof}
\section{Behaviour of $R$--groups for $\SU_n$ and its inner forms} \label{R-groups for SU(n) and its inner forms}
In this section, we prove Arthur's conjecture for both quasi-split special unitary groups $\SU_n$ and its inner forms.
Furthermore, we study the behavior of $R$-groups within $L$-packets and between inner forms of $\SU_n.$
We will use the notation in Sections \ref{structure}, \ref{weyl group actions}, and \ref{R-groups for U(n) and its inner forms}.
\subsection{$L$-packets for $\SU_n$ and its inner forms} \label{L-packets}
We discuss tempered $L$-packets of $\bG=\SU_n$ and its $F$-inner form $\bG'.$
It is natural to construct $L$-packets for $\bG$ and $\bG',$ 
by restricting $L$-packets for $\tbG=\U_n$ and its $F$-inner form $\tbG'$ 
which has been done by Rogawski \cite{rog90}, Mok \cite{mok13}, and Kaletha-Minguez-Shin-White \cite{kmsw14} (see Section \ref{LLC for unitary}).
Thus, given $\phi \in \Phi(G),$ from \cite[Th\'{e}or\`{e}m 8.1]{la85}, there exists a lifting $\tphi \in \Phi(\tG)$ such that
\[
\phi = \tphi \circ pr,
\]
where $pr$ is the projection  $\widehat{\tG} \twoheadrightarrow \widehat{G}$ in \eqref{exact of L-gp of G tG}.
Note that  the homomorphism $pr$ is compatible with $\Gamma$-actions on $\widehat{\tG}$ and $\widehat{G}$ (see Section \ref{L-groups}) and the lifting $\tphi \in \Phi(\tG)$ is chosen uniquely up to a $1$-cocycle of $W_F$ in $\widehat{(\tG/G)}$ (see \cite[Section 7]{la85} and \cite[Theorem 3.5.1]{chaoli}). 

For our purpose, we are interested in $\phi \in \Phi_{\temp}(G)$ and the lifting $\tphi$ lies in $\Phi_{\temp}(\tG).$
Using the $L$-packet $\Pi_{\tphi}(\tG)$ for $\tphi \in \Phi_{\temp}(\tG)$ in Section \ref{LLC for unitary}, we construct an $L$-packet  $\Pi_{\phi}(G)$ for $\phi \in \Phi_{\temp}(G)$ as the set of isomorphism classes of irreducible constituents in the restriction from $\tG$ to $G$ as follows:
\[
\Pi_{\phi}(G) := \{ \sigma \hookrightarrow {\Res}^{\tG}_{G}(\tsigma),~~ \tsigma \in \Pi_{\tphi}(\tG) \} / \s.
\]
Likewise, given $\phi \in \Phi_{\disc}(M),$ we construct an $L$-packet $\Pi_{\phi}(M)$ for $\phi \in \Phi_{\disc}(M)$ as follows:
\[
\Pi_{\phi}(M) := \{ \sigma \hookrightarrow {\Res}^{\tM}_{M}(\tsigma),~~ \tsigma \in \Pi_{\tphi}(\tM) \} / \s,
\]
where $\tphi$ lies in $\Phi_{\disc}(\tM)$ such that $
\phi = \tphi \circ pr,$ with the projection $pr: \widehat{\tM} \twoheadrightarrow \widehat{M}$ in \eqref{exact of L-gp of M tM}.

All the above arguments apply verbatim to $F$-inner forms $\tbG'$ and their $F$-Levi subgroups $\tbM'.$
\subsection{Arthur conjecture for $\SU_n$ and its inner forms} \label{arthur conj}
The purpose of the section is to prove Arthur's conjecture, predicted in \cite{art89ast}, for $\bG$ and $\bG'.$ 
\begin{thm}  \label{arthur conj}
Given $\phi \in \Phi_{\disc}(M),$ $\sigma \in \Pi_{\phi}(M),$ and $\sigma' \in \Pi_{\phi}(M'),$ we have
\[
R_{\sigma} \s R_{\phi, \sigma} ~~ \text{and} ~~R_{\sigma'} \s R_{\phi, \sigma'}.
\]
\end{thm}
The rest of the section is devoted to the proof of Theorem \ref{arthur conj}.
Since all the following techniques apply to both $\bM$ and $\bM',$ we shall state the proof for $\bM.$

Let $\phi \in \Phi_{\disc}(M)$ and $\sigma \in \Pi_{\phi}(M)$ be given.
Identifying $W_\phi$ with a subgroup of $W_M$ (see Section \ref{section for def of R}), we note that $W(\sigma) \subset W_\phi,$ which implies that
\begin{equation} \label{reduction for W}
W(\sigma) = W_{\sigma, \phi} = W_\phi \cap W(\sigma).
\end{equation}

Denote by $\tphi \in \Phi_{\disc}(\tM)$ a lifting of $\phi$ as in Section \ref{L-packets}.  
For any $\tsigma \in \Phi_{\tphi}(\tM),$ we have
\[
W^{\circ}_{\tsigma} = W^{\circ}_{\sigma},
\]
since the Plancherel measure is compatible with restriction (see \cite[Proposition 2.4]{choiy1}, \cite[Lemma 2.3]{go06}, for example) and $\Phi(P, A_M) = \Phi(\tP, A_{\tM})$ (see Section \ref{basic notation}).
Using \eqref{identity of W'}, we have
\begin{equation} \label{reduction for W'}
W^{\circ}_{\sigma} = W^{\circ}_{\tphi}.
\end{equation}
To prove Theorem \ref{arthur conj}, from \eqref{reduction for W} and \eqref{reduction for W'}, it is thus enough to show that
\begin{equation} \label{last arg for Arthur conj}
W^{\circ}_{\tphi} = W^{\circ}_{\phi}.
\end{equation}
Indeed, if so, then we have
\[
W^\circ_{\sigma, \phi} \overset{\text{definition}}{=} W^\circ_{\phi} \cap W(\sigma) \overset{\eqref{last arg for Arthur conj}}{=} W^\circ_{\tphi} \cap W(\sigma) \overset{\eqref{reduction for W'}}{=} W^\circ_{\sigma} \cap W(\sigma) = W^\circ_\sigma.
\]

In what follows, we prove \eqref{last arg for Arthur conj}.
From \cite[Chapter 3.4]{mok13} we set a maximal torus $T_{\tphi}$ in $C_{\tphi}(\widehat{\tG})^{\circ}$ to be the identity component
\[
A_{\widehat \tM} = (Z(\widehat{\tM})^{\Gamma})^{\circ}
\]
of the $\Gamma$--invariants of the center $Z(\widehat{\tM}).$ 
So, we have $\widehat \tM = Z_{\widehat \tG}(T_{\tphi}).$ 
Likewise, we set $T_{\phi}= A_{\widehat M} \subset C_{\phi}(\widehat{G})^{\circ},$ and  $\widehat M = Z_{\widehat G}(T_{\phi}).$
Here, we put
\[
\bar C_{\tphi}^\circ := C_{\tphi}(\widehat{\tG})^\circ/Z(\widehat \tG)^{\Gamma}, ~~ \bar T_{\tphi} := T_{\tphi}/Z(\widehat \tG)^{\Gamma}.
\]
Note that $\bar C_{\tphi}^\circ \subset C_{\phi}(\widehat{G})^\circ \subset \widehat G$ and $\bar T_{\tphi} \subset T_{\phi} \subset \widehat M.$ 
Using an equality in \cite[p.64]{mok13}, we have
\[
W_{\tphi}^\circ \overset{\text{definition}}{=} N_{C_{\tphi}(\widehat{\tG})^\circ}(T_{\tphi}) = N_{\bar C_{\tphi}^\circ}(\bar T_{\tphi}).
\]
Thus, it suffices to show that 
\[
N_{\bar C_{\tphi}^\circ}(\bar T_{\tphi}) = N_{C_{\phi}(\widehat{G})^\circ}(T_{\phi}).
\]

With above notation, the following lemma holds.
\begin{lm} \label{lm for identity comp}
\[
\bar C_{\tphi}^\circ =  C_{\phi}(\widehat{G})^\circ
\]
\end{lm}
\begin{proof}
We first consider the following commutative diagram
\[
\begin{CD}
@.   @.1  @.  1 @.
\\
@.      @.        @VVV   @VVV  @.\\
1 @>>> Z(\widehat{\tG})^\Gamma  @>>> C_{\tphi}(\widehat{\tG})^\circ @>>> \bar C_{\tphi}^\circ := C_{\tphi}(\widehat{\tG})^\circ/ Z(\widehat{\tG})^\Gamma @>>> 1 
\\
@.      @|          @VVV   @VVV  @.
\\
1 @>>> Z(\widehat{\tG})^\Gamma @>>> C_{\tphi}(\widehat{\tG}) @>>> C_{\tphi}(\widehat{\tG})/Z(\widehat{\tG})^\Gamma @>>> 1 \\
@.      @.          @VVV   @VVV  @.\\
@.   @. S_{\tphi}:=\pi_0(C_{\tphi}(\widehat{\tG})) @= S_{\tphi}:=\pi_0(C_{\tphi}(\widehat{\tG})) @>>> 1 .
\end{CD}
\]
It follows from the right vertical exact sequence above that
\[
\bar C_{\tphi}^\circ = C_{\tphi}(\widehat{\tG})^\circ/ Z(\widehat{\tG})^\Gamma \subset (C_{\tphi}(\widehat{\tG})/Z(\widehat{\tG})^\Gamma)^\circ \subset C_{\tphi}(\widehat{\tG})/Z(\widehat{\tG})^\Gamma.
\]
Note that the index 
$
[\bar C_{\tphi}^\circ : C_{\tphi}(\widehat{\tG})/Z(\widehat{\tG})^\Gamma]
$
is finite and so is $
[\bar C_{\tphi}^\circ : (C_{\tphi}(\widehat{\tG})/Z(\widehat{\tG})^\Gamma)^\circ].
$ Since $\bar C_{\tphi}^\circ$ is connected due to the isomorphism $C_{\tphi}(\widehat{\tG})^\circ/ Z(\widehat{\tG})^\Gamma \s (C_{\tphi}(\widehat{\tG})^\circ \cdot Z(\widehat{\tG}))/ Z(\widehat{\tG}),$ we have
\begin{equation} \label{an eqaulity in iendtity comp}
\bar C_{\tphi}^\circ = (C_{\tphi}(\widehat{\tG})/Z(\widehat{\tG})^\Gamma)^\circ.
\end{equation}
From the proof of \cite[Lemma 5.3.4]{chaoli} and the fact that $C_{\tphi}(\widehat{\tG})/Z(\widehat{\tG})^\Gamma \subset C_{\phi}(\widehat{G}),$ we have the exact sequence 
\[
1 \longrightarrow C_{\tphi}(\widehat{\tG})/Z(\widehat{\tG})^\Gamma \longrightarrow C_{\phi}(\widehat{G}) \longrightarrow X^{\tG}(\tphi) \longrightarrow 1,
\]
where $X^{\tG}(\tphi) := \{ \bold a \in H^1(W_F,\widehat{\U_1}) : \bold a \tphi \s \tphi  \}.$ Since $X^{\tG}(\tphi)$ is finite, we have
\[
(C_{\tphi}(\widehat{\tG})/Z(\widehat{\tG})^\Gamma)^\circ = C_{\phi}(\widehat{G})^\circ.
\]
Therefore, from \eqref{an eqaulity in iendtity comp}, Lemma \ref{lm for identity comp} is proved.
\end{proof}

Combining Lemmas \ref{gamma-inv of centers} and \ref{lm for identity comp}, we have proved \eqref{last arg for Arthur conj}.
This completes the proof of Theorem \ref{arthur conj}.
\subsection{Invariance of $R$-groups within $L$-packets and between inner forms} \label{behavior of Rgps}
Let $\phi \in \Phi_{\disc}(M)$ be given. In this section, we will discuss the behavior of Knapp-Stein $R$-groups within $L$-packets and between inner forms.
We first provide the following theorem that is crucial to the invariance (Theorems \ref{invariance within L-packet} and \ref{invariance 2 within L-packet}).
\begin{thm} \label{equality on W}
Fix a lift $\tphi \in \Phi_{\disc}(\tM)$ and $\tsigma \in \Pi_{\tphi}(\tM).$ Then, we have
\[
W(\sigma) \s 
\{
w \in W_M : {^w}\tsigma \s \tsigma \lambda ~~ \text{for some }~ \lambda \in (\tM/M)^D
\}.
\]
\end{thm}
\begin{proof}
It suffice to consider the case of $m \geq 2,$ since the equality is already true for the case  $m =1$ due to \cite[Lemma 2.5]{go06}.
Let $m \geq 2.$ We then recall from \eqref{form of bM} that
\[
M \s ({\GL}_{n_1}(E) \times \cdots \times {\GL}_{n_k}(E)) \rtimes {\SU}_m(F).
\]
Since $\tsigma$ is of the form 
\[
\tsigma_1 \otimes \tsigma_2 \otimes \cdots \otimes \tsigma_k \otimes \ttau,
\]
we set $\tsigma = \tsigma_0 \otimes \ttau.$
Denote by $\tau$ is an irreducible constituent in $\Res^{\U_m}_{\SU_m}(\ttau).$
Write $V_{\tsigma},$ $V_{\tsigma_0}=\otimes_{i=1}^kV_{\tsigma_i},$ $V_{\ttau},$ $V_\tau$ for the corresponding representation $\CC$-vector spaces.

Given $(g,h)\in M$ with $g = (g_1, \cdots, g_k) \in {\GL}_{n_1}(E) \times \cdots \times {\GL}_{n_k}(E)$ and $h \in {\SU}_m(F),$ $v'_0 = v_1 \otimes \cdots \otimes v_k \in V_{\tsigma_0},$ and $v_0 \in V_\tau,$
we define a representation $\sigma^*$ of $M$ 
\begin{equation} \label{def sigma*}
\sigma^*((g, h))(v'_0 \otimes v_0) =\tsigma_0(g)v'_0 \otimes \ttau(\alpha_m(g)^{-1})\tau(h)v_0.
\end{equation}
It turns out from \cite[p.353]{go06} that
$\sigma^*$ is an irreducible constituent in $\Res^{\U_n}_{\SU_n}(\tsigma).$

In what follows, using the idea in the proof of \cite[Proposition 2.9]{go06}, we will show
\begin{equation} \label{W*=}
W(\sigma^*) = \{
w \in W_M : {^w}\tsigma \s \tsigma \lambda ~~ \text{for some }~ \lambda \in (\tM/M)^D
\}.
\end{equation}
The inclusion $\subset$ is obvious. 
Suppose that ${^w}\tsigma \s \tsigma \lambda$ for some $\lambda \in (\tM/M)^D.$ 
From \eqref{desc WM}, we set $w=sc$ with $s\in S_k$ and $c \in \ZZ_2^k.$ 
From Sections \ref{WM-actions on Levis} and \ref{WM-actions on rep of Levis}, we have
\[
{^w}\tsigma \s \bigotimes_{i=1}^k\varepsilon_i(\tsigma_{s(i)}) \otimes \ttau,
\]
where $\varepsilon_i$ is either $\varepsilon$ or trivial.
Since ${^w}\tsigma \s \tsigma \lambda,$ we have
$\varepsilon_i(\tsigma_{s(i)})  \s \tsigma_i \lambda$
for each $i,$ and $\ttau \s \ttau \lambda.$ 
We fix intertwining maps $T_i:V_{\tsigma_{s(i)}}\rightarrow V_{\tsigma_i}$ with $\tsigma_i \lambda T_i=T_i \varepsilon_i(\tsigma_{s(i)}),$ 
and $T_\lambda:V_{\ttau} \rightarrow V_{\ttau}$ with $\ttau \lambda T_{\lambda}=T_{\lambda} \ttau.$
Write
\[
T=\bigotimes_{i=1}^k T_i \otimes T_{\lambda}.
\]
Note that $\tsigma T = T {^w}\tsigma.$
Then, from the definition \eqref{def sigma*} of $\sigma^*,$ we have
\begin{align*}
T{^w}\sigma^*(g,h)(v'_0 \otimes v_0) 
&= \bigotimes_{i=1}^k T_i \varepsilon_i(\tsigma_{s(i)})(g_i)v_i \otimes T_{\lambda} \ttau (\alpha_m({^w}g)^{-1})\tau(\alpha_m(g)^{-1}\alpha_m({^w}g)h)v_0 \\
&= \bigotimes_{i=1}^k T_i \varepsilon_i(\tsigma_{s(i)})(g_i)v_i \otimes T_{\lambda} \ttau (\alpha_m(g)^{-1})\tau(h)v_0 \\
&= \bigotimes_{i=1}^k 
\tsigma_i(g_i)\lambda(g_i)T_i(v_i) \otimes  \ttau(\alpha_m(g)^{-1}) \lambda(\alpha_m(g)^{-1}) \tau(h)T_{\lambda}(v_0) \\
&= \sigma^*\lambda(g,h)T(v_0'\otimes v_0)\\
&=\sigma^*T(v_0'\otimes v_0),
\end{align*}
since 
\[
\lambda(g,h) \overset{\eqref{imp iso}}{=} \lambda( 
\left( \begin{array}{ccc}
 g& & \\
 & \alpha_m(g)^{-1}h & \\
 & & \varepsilon(g) \\ 
 \end{array} \right)) = 1.
\]
Thus, we have $w \in W(\sigma^*),$ which implies
\eqref{W*=}. 
From the proof of \cite[Theorem 4.19]{chgo12}, we note
\begin{equation*} \label{W*}
W(\sigma) \s W(\sigma^*),
\end{equation*}
since both are in the same restriction $\Res^{\U_n}_{\SU_n}(\tsigma).$
Therefore, we proved the theorem.
\end{proof}
\begin{rem}
The difference between $W(\sigma)$ and $W(\tsigma)$ turns out to be whether the twisting by a character in $(\tM/M)^D$ is allowed or not.
We also refer to \cite[Theorem 3.7]{go06} for another description of this difference.
\end{rem}
\begin{rem}
We note that Theorem \ref{equality on W} holds for $F$-inner forms $\bM'$ as well. 
To be precise, given $\sigma' \in \Pi_{\phi}(M'),$ we fix $\tsigma' \in \Pi_{\tphi}(\tM').$ 
Then, we have
\[
W(\sigma') \s 
\{
w \in W_{M'} : {^w}\tsigma' \s \tsigma' \lambda ~~ \text{for some }~ \lambda \in (\tM'/M')^D
\}.
\]
\end{rem}
\begin{thm} \label{invariance within L-packet}
Given $\sigma_1, ~ \sigma_2 \in \Pi_{\phi}(M),$ we have
\[
R_{\sigma_1} \s R_{\sigma_2}.
\]
Furthermore, given $\sigma'_1, ~ \sigma'_2 \in \Pi_{\phi}(M'),$ we have
\[
R_{\sigma'_1} \s R_{\sigma'_2}.
\]
\end{thm}
\begin{proof}
Since $W^{\circ}(\sigma_1) = W^{\circ}(\sigma_2),$ it is enough to show that 
\[
W(\sigma_1) \s W(\sigma_2).
\]
Fix $\tsigma^1$ and $\tsigma^2$  in $\Pi_{\tphi}(\tM)$ such that $\sigma_1 \hookrightarrow \Res^{\U_n}_{\SU_n}(\tsigma^1)$ and $\sigma_2 \hookrightarrow \Res^{\U_n}_{\SU_n}(\tsigma^2).$ 
Suppose that ${^w}\tsigma^1 \s \tsigma^1 \lambda$ for some $\lambda \in (\tM/M)^D.$
From \eqref{desc WM}, we write $w=sc$ with $s\in S_k$ and $c \in \ZZ_2^k.$ 
Set $\tsigma^1 \s \tsigma^1_{1} \otimes \tsigma^1_{2} \otimes \cdots \otimes \tsigma^1_{k} \otimes \ttau_1.$
Then $\tsigma^1_{i} \s \varepsilon_i(\tsigma^1_{s(i)}) \lambda$
for each $i,$ and $\ttau_1 \s \ttau_1 \lambda.$ 
Set 
\[
\lambda' := \left\{ \begin{array}{ll}
         \lambda, & \mbox{on}~~{\GL}_{n_1}(E) \times \cdots \times {\GL}_{n_k}(E),\\
        \mathbbm{1}, & \mbox{on} ~~{\U}_m(F),
        \end{array} \right. 
\]
which is a character on $\tM/M.$
Moreover, $\lambda'$ satisfies
\[
{^w}\tsigma^2 \s \tsigma^2 \lambda',
\]
which implies, from Theorem \ref{equality on W}, that 
$W(\sigma_1) \hookrightarrow W(\sigma_2).$
In the same manner, one can verify $W(\sigma_2) \hookrightarrow W(\sigma_1).$
Thus, we have $R_{\sigma_1} \s R_{\sigma_2}.$
Since the method is the same, we omit the proof for $R_{\sigma'_1} \s R_{\sigma'_2}.$
\end{proof}
\begin{rem}
The trivial character $\mathbbm{1}$ in the definition of $\lambda'$ can be replaced by 
a character $\lambda_0$ on $\U_m(F)$ such that $\ttau_2 \s \ttau_2 \lambda_0,$ where $\ttau_2$ is the representation of $\U_m(F)$ in the decomposition $\tsigma^2 \s \tsigma^2_{1} \otimes \tsigma^2_{2} \otimes \cdots \otimes \tsigma^2_{k} \otimes \ttau_2.$
\end{rem}
\begin{thm} \label{invariance 2 within L-packet}
Given $\sigma \in \Pi_{\phi}(M)$ and $\sigma' \in \Pi_{\phi}(M),$ we have
\[
R_{\sigma} \s R_{\sigma'}.
\]
\end{thm}
\begin{proof}
The proof is similar to Theorem \ref{invariance within L-packet}.
\end{proof}
\begin{rem} \label{final remark}
Due to \cite[Theorem 3.7]{go95}, Theorem \ref{invariance 2 within L-packet} shows that the Knapp-Stein $R$-group $R_{\sigma'}$ is of the form 
\[
\Gamma_{\sigma'} \ltimes \ZZ^d,
\]
for some subgroup $\Gamma_{\sigma'}$ in $R_{\sigma'}$ defined in \textit{loc. cit} and some integer $d$ in Corollary \ref{r-groups for G'=bc}. 
This implies that $R_{\sigma'}$ is in general non-abelian from the argument in \cite[Remark p.359]{go95}.
\end{rem}

\end{document}